\documentclass[11pt, oneside]{amsart}
\usepackage{mathtools}
\usepackage{wrapfig}
\usepackage{tikz}
\usepackage{tikz-cd}
\usepackage{comment}
\usepackage{mathrsfs}
\usepackage{tabularx, hyperref}
\usepackage{amssymb} \usepackage{amsfonts} \usepackage{amsmath}
\usepackage{amsthm} \usepackage{epsfig, subfig}
\usepackage{amscd, amsxtra, latexsym}
\usepackage{epsfig,  graphicx, psfrag}
\usepackage[all]{xy}
\usepackage{caption}
\usepackage{enumerate}
\usepackage{color}

\newtheorem{lemma}{Lemma}[section]
\newtheorem{thm}[lemma]{Theorem}
\newtheorem{prop}[lemma]{Proposition}

\newtheorem*{prop*}{Proposition}
\newtheorem{prop_intro}{Proposition}

\newtheorem{thm_intro}[prop_intro]{Theorem}

\theoremstyle{definition}
\newtheorem{rem_intro}[prop_intro]{Remark}

\newtheorem{defn}[lemma]{Definition}

\newtheorem{example}[lemma]{Example}
\newtheorem{rem}[lemma]{Remark}

\theoremstyle{definition}

\definecolor{darkgreen}{cmyk}{1,0,1,.2}

\DeclareMathOperator{\inte}{int}
\DeclareMathOperator{\id}{id}
\DeclareMathOperator{\mult}{mult}

\newcommand{\lf}{\text{lf}}

\renewcommand{\phi}{\varphi}
\newcommand{\invlim}{\underset{\longleftarrow}{\lim}}

\begin{document}

\title[]{Simplicial volume of manifolds with amenable fundamental group at infinity}

\author[Giuseppe Bargagnati]{Giuseppe Bargagnati}
\address{Dipartimento di Matematica, Universit\`a di Pisa, Italy}
\email{g.bargagnati@studenti.unipi.it}

\subjclass{57N16, 55N10 (Primary); 53C23, 57N65 (Secondary)}



\keywords{}

\begin{abstract}
    We show that for $n \neq 1,4$ the simplicial volume of an inward tame triangulable open $n$-manifold $M$ with amenable fundamental group at infinity at each end is finite; moreover, we show that if also $\pi_1(M)$ is amenable, then the simplicial volume of $M$ vanishes. We show that the same result holds for finitely-many-ended triangulable manifolds which are simply connected at infinity. 
\end{abstract}

\maketitle

\section{Introduction}
The simplicial volume is a homotopy invariant of manifolds introduced by Gromov in his pioneering article \cite{Gro82} (see Section \ref{simvolamc:sec} for the precise definition). Although the definition of this invariant is completely homotopic in nature, Gromov himself remarked that it is strongly related to the geometric structures that a manifold can carry: for example, he proved that the simplicial volume of an oriented closed connected hyperbolic manifold is strictly positive, and it is proportional to its Riemannian volume. On the other hand, the vanishing of the simplicial volume is implied by amenability conditions on the manifold: in particular, if an oriented closed connected manifold $M$ has an amenable fundamental group, then it has vanishing simplicial volume (this follows from the vanishing of the bounded cohomology of amenable groups, see e.g. \cite[Section 3.1]{Gro82}, \cite[Corollary 7.12]{frigerio:book}). In the context of open manifolds the situation is different; in fact, an open manifold with an amenable fundamental group has either vanishing or infinite simplicial volume (\cite[Corollary 5.18]{Lothesis}). A finiteness criterion for the simplicial volume of tame manifolds, i.e. manifolds which are homeomorphic to the interior of a compact manifold with boundary, was given in \cite[Theorem 6.1]{Lothesis}, \cite[Theorem 6.4]{Loeh}. It states that, given a compact manifold with boundary $(W, \partial W)$, the finiteness of the simplicial volume of $W^{\circ}$ is equivalent to the $\ell^1$-invisibility of $\partial W$ (i.e., to the fact that the fundamental class of $\partial W$ is null in $\ell^1$-homology). A consequence of this result is that if each connected component of $\partial W$ has an amenable fundamental group, then the simplicial volume of $W^{\circ}$ is finite. 

The fundamental group at infinity (see Section \ref{fgi:sec} for the definition) is an algebraic tool that is useful to study open manifolds. If $(W, \partial W)$ is an oriented compact connected manifold with connected boundary, the fundamental group at infinity of $W^{\circ}$ is isomorphic to $\pi_1(\partial W)$. This concept allows us to generalize the result of L\"oh to manifolds which are not necessarily the interior of a compact manifold with boundary. Since the fundamental group at infinity is defined as an inverse limit, it is endowed with a natural topology. We relate the amenability of this group as a topological group to the finiteness/vanishing of the simplicial volume. In particular, we focus on two classes of manifolds:

\begin{itemize}
    \item \textit{inward tame manifolds}, i.e. manifolds in which the complement of each compact subset can be shrunk into a compact within itself (see Section \ref{fgi:sec} for the precise definition);
    \item manifolds which are \textit{simply connected at infinity}, i.e. such that for any compact subset $C \subset M$ there exists a larger compact subset $D \supset C$ such that every loop in $M \setminus D$ is trivial in $M \setminus C$.
\end{itemize}

The main purpose of this paper is to prove the following results.

\begin{thm_intro}\label{teor1}
Let $n \geq 5$ be a natural number and let $M$ be a connected triangulable inward tame $n$-manifold with amenable fundamental group at infinity at each end. Then the simplicial volume of $M$ is finite.
\end{thm_intro}

\begin{thm_intro}\label{teor2}
Let $n \geq 5$ be a natural number and let $M$ be a connected triangulable inward tame $n$-manifold with amenable fundamental group and amenable fundamental group at infinity at each end. Then the simplicial volume of $M$ vanishes.
\end{thm_intro}

\begin{thm_intro}\label{teor3}
Let $n\geq 5$ be a natural number, and let $M$ be an open triangulable finitely-many-ended $n$-manifold. Let us suppose that $M$ is simply connected at infinity. Then the simplicial volume of $M$ is finite. Moreover, if $\pi_1(M)$ is amenable, we have that the simplicial volume of $M$ vanishes.
\end{thm_intro}

\begin{rem_intro}
In dimension $n=2$ there are no inward tame manifolds which are not tame; remarkably, combining \cite{TUCKER1974267} with the Poincar\'e Conjecture (\cite{Per02}, \cite{Per03}, \cite{Per03bis}, \cite{CaoZhu}, \cite{MorganTian}, \cite{francesi}), the same holds in dimension $n=3$. In particular this implies that Theorem \ref{teor1}, Theorem \ref{teor2} and Theorem \ref{teor3} hold also if $n=2,3$ thanks to previous results of L\"oh \cite[Theorem 6.1]{Lothesis}, \cite[Theorem 6.4]{Loeh}.
\end{rem_intro}

The main ingredients for the proofs are two results of \cite{FM19}. These results give sufficient conditions for an open cover to imply the finiteness (respectively, the vanishing) of the simplicial volume.

Under our hypotheses, we deduce the existence of such covers from the amenability of the fundamental group at infinity at each end. In order to do this, we use a result from \cite{Gu00}, which is stated at the end of Section \ref{fgi:sec}. 

\subsection*{Plan of the paper}
In Section \ref{fgi:sec}, we recall the algebra of inverse sequences, we define the fundamental group at infinity and we specify the terminology regarding open manifolds, in particular we focus on the classes of tame manifolds, inward tame manifolds and manifolds which are simply connected at infinity. In Section \ref{simvolamc:sec}, we define the simplicial volume and we state the results from \cite{FM19} that will be the key for the proof of Theorems \ref{teor1}, \ref{teor2} and \ref{teor3}, to which Section \ref{proof:sec} is devoted.

\subsection*{Acknowledgments}
I thank my advisor Roberto Frigerio for helpful comments and discussions. I thank Dario Ascari and Francesco Milizia for useful conversations.

\section{Fundamental group at infinity}\label{fgi:sec}

Following \cite{Gu00}, we start by recalling the algebra of inverse sequences, that will be useful for the definition of the fundamental group at infinity.

Let

\[G_0 \overset{f_1}{\longleftarrow} G_1 \overset{f_2}{\longleftarrow} G_2 \overset{f_3}{\longleftarrow} \ldots\]
be an inverse sequence of groups and homomorphisms (such a sequence will be denoted by $\{G_i, f_i\}$ henceforth). Let $i,j$ be two natural numbers such that $i>j$; after denoting the composition $f_{j}\circ f_{j+1} \ldots \circ f_{i}$ by $f_{ij}$, we can naturally define a subsequence of $\{G_i, f_i\}$ as follows:
\[G_{i_0} \overset{f_{i_1i_0}}{\longleftarrow} G_{i_1} \overset{f_{i_2i_1}}{\longleftarrow} G_{i_2} \overset{f_{i_3i_2}}{\longleftarrow} \ldots \ .\]
Two sequences $\{G_i, f_i\}$ and $\{H_i, g_i\}$ are said to be pro-isomorphic if there exists a commuting diagram of the form
\begin{center}
\begin{tikzcd}
G_{i_0} & & \arrow[ll, "f_{i_1i_0}"] G_{i_1} \arrow[dl] & & G_{i_2} \arrow[ll, "f_{i_2i_1}"] \arrow[dl] &  \ldots \arrow[l] \\
& H_{j_0} \arrow[ul] & & \arrow[ll, "g_{j_1j_0}"] H_{j_1} \arrow[ul] & & H_{j_2} \arrow[ll, "g_{j_2j_1}"] \arrow[ul] & \ldots \arrow[l]
\end{tikzcd} 
\end{center}
for some subsequences $\{G_{i_k}, f_{i_ki_{k-1}}\}$ and $\{H_{j_k}, g_{j_kj_{k-1}}\}$ of $\{G_i, f_i\}$ and $\{H_j, g_j\}$ respectively.

The inverse limit of a sequence $\{G_i, f_i\}$ is defined as

\[\invlim \{G_i, f_i\}=\Big\{(g_0, g_1, \ldots) \in \prod_{i \in \mathbb{N}}G_i \big| f_i(g_i)=g_{i-1}\Big\},\]
with the group operation defined componentwise; this turns out to be a subgroup of $\prod_{i \in \mathbb{N}}G_i$. For each $i \in \mathbb{N}$ we have a natural projection homomorphism $p_i: \invlim \{G_i, f_i\} \rightarrow G_i$.

It can be easily seen that inverse limits of pro-isomorphic sequences are isomorphic, while we remark that two sequences with isomorphic inverse limit are not necessarily pro-isomorphic; for example the trivial sequence
\[1 \longleftarrow 1 \longleftarrow 1 \longleftarrow \ldots\]
and the following one
\[\mathbb{Z} \overset{\cdot 2}{\longleftarrow}\mathbb{Z} \overset{\cdot 2}{\longleftarrow}\mathbb{Z} \overset{\cdot 2}{\longleftarrow} \ldots\]
have both inverse limit isomorphic to the trivial group, but it is straightforward that they are not pro-isomorphic (see e.g. \cite[Subsection 3.4]{10.1007/978-3-319-43674-6_3}).

We distinguish some particular classes of sequences; denoting by $\{G, \id_G\}$ a sequence $\{G_i, f_i\}$ such that $G_i \simeq G$ and $f_i=\id_G$ for all $i\in \mathbb{N}$, we say that:
\begin{itemize}
    \item a sequence pro-isomorphic to one of the form $\{1, \id \}$ is called \textit{pro-trivial};
    \item a sequence pro-isomorphic to one of the form $\{H, \id_H\}$ is called \textit{stable};
    \item a sequence pro-isomorphic to one whose maps are all surjective is called \textit{semistable}.
\end{itemize}

In the case of stable sequences, we remark that $H$ is well defined up to isomorphism, since the inverse limit of the sequence is isomorphic to $H$.

\begin{lemma}\label{lemamn}
Let $\{G_i\}_{i \in \mathbb{N}}$ be discrete groups and let $f_i:G_{i}\rightarrow G_{i-1}$ be a group homomorphism for each $i \in \mathbb{N}_{>0}$. Let us suppose that $G= \invlim \{G_i, f_i\}$ equipped with the limit topology is an amenable group (as a topological group) and that for every $i>0$ the homomoprhism $f_i$ is surjective. Then $G_i$ is amenable for each $i\in \mathbb{N}$. 
\end{lemma}

\begin{proof}
By definition of inverse limit, if the $f_i$'s are surjective we have that the projections $p_i:G \rightarrow G_i$ are surjective too; by definition of the limit topology on $G$, we have that the projections are also continuous. Then each $G_i$ is a continuous image of an amenable group, so it is amenable \cite[Lemma 4.1.13]{Zimmer1984}.
\end{proof}

\begin{rem}
As remarked in \cite[Proposition 1.1]{GM96}, the limit topology on an inverse limit of a sequence of discrete groups is discrete if and only if the sequence is stable.
\end{rem}

Following again \cite{Gu00}, we recall the definition of fundamental group at infinity of a manifold.
Let $M$ be a triangulable manifold (henceforth, all manifolds will be triangulable) without boundary, and let $A \subset M$ be any subset; we say that a connected component of $M \setminus A$ is bounded if it has compact closure, otherwise we say that it is unbounded. A \textit{neighborhood of infinity} is a subset $N \subset M$ such that $M \setminus N$ is bounded. We say that a neighborhood of infinity is \textit{clean} if it satisfies the following properties:
\begin{itemize}
    \item $N$ is closed in $M$;
    \item $N$ is a submanifold of codimension $0$ with bicollared boundary.
\end{itemize}
It can be easily shown that every neighborhood of infinity contains a clean neighborhood of infinity.\\
Let $k$ be a natural number. A manifold $M$ is said to have $k$ ends if there exists a compact subset $C \subset M$ such that for any other compact subset $D$ such that $C \subset D$, $M \setminus D$ has exactly $k$ unbounded components; this implies that $M$ contains a clean neighborhood of infinity with $k$ connected components. If such $k$ does not exist, we say that the manifold $M$ has infinitely many ends. 

Let us suppose that $M$ is a $1$-ended manifold without boundary. A clean neighborhood of infinity is called a $0$-neighborhood of infinity if it is connected and has connected boundary. A sequence of nested neighborhoods of infinity $U_0 \supset U_1 \supset U_2 \supset \ldots$ is \textit{cofinal} if $\cap_{i=0}^{+\infty}U_i= \varnothing$.

Let $U_0 \supset U_1 \supset U_2 \supset \ldots$ be a cofinal sequence of nested $0$-neighborhoods of infinity; a \textit{base-ray} is a proper map  $\rho:[0,+ \infty) \rightarrow M$. For each $i\in \mathbb{N}$ let $p_i:=\rho(i)$; up to reparametrization, we can assume that $p_i\in U_i$ and $\rho([i, +\infty))\subset U_i$. The inclusions $U_{i-1} \hookleftarrow U_{i}$ along with the change of basepoints maps induced by $\rho|_{[i-1, i]}$ give rise to the following sequence
\[\pi_1(U_0, p_0) \overset{\lambda_1}{\longleftarrow} \pi_1(U_1, p_1) \overset{\lambda_2}{\longleftarrow} \pi_1(U_2, p_2) \overset{\lambda_3}{\longleftarrow} \ldots .\]
We denote the inverse limit of this sequence with $\pi_1^{\infty}(M, \rho)$. If the sequence is semistable, it can be shown that its pro-equivalence class does not depend neither on the base-ray $\rho$ nor on the choice of the sequence of neighborhoods of infinity (see e.g. \cite[Section 16.1]{Geoghegan2008}). In this case, we can define the fundamental group at infinity of $M$ as
\[\pi_1^{\infty}(M):= \invlim (\pi_1(U_i, p_i), \lambda_i).\]

If $M$ has more than one end, let $U_0 \supset U_1 \supset U_2 \supset \ldots$ be a cofinal sequence of nested clean neighborhoods of infinity; we say that two base-rays $\rho_1, \rho_2$ \textit{determine the same end} if their restrictions $\rho_1|_{\mathbb{N}}$ and $\rho_2|_{\mathbb{N}}$ are properly homotopic. In the case of multiple ends, we need to specify the base-ray to define the fundamental group at infinity \textit{of each end}; however, if the sequence that defines the fundamental group at infinity at each end is semistable we have that if $\rho_1$ and $\rho_2$ determine the same end, then $\pi_1^{\infty}(M, \rho_1)\simeq \pi_1^{\infty}(M, \rho_2)$ (see again \cite[Section 16.1]{Geoghegan2008}).

With a similar procedure, we can define the homology at infinity of a $1$-ended manifold $M$: let $U_0 \supset U_1 \supset U_2 \supset \ldots$ be a cofinal sequence of nested $0$-neighborhoods of infinity and let us denote the map induced by the inclusion $U_{j+1}\hookrightarrow U_j$ on the $k$-th homology groups with coefficients in a ring $R$ by $\mu_j$. Then the $k$-th homology at infinity of $M$ with $R$-coefficients is defined as the inverse limit of the following sequence:

\[H_k(U_0; R) \overset{\mu_1}{\longleftarrow} H_k(U_1; R) \overset{\mu_2}{\longleftarrow} H_k(U_2, R) \overset{\mu_3}{\longleftarrow} \ldots ,\]
and is denoted by $H_k^{\infty}(M; R)$. If this sequence is semistable, it can be shown that the limit does not depend on the choice of the neighborhoods of infinity. In the multiple-ended case, we can define the homology at infinity \textit{at each end} similarly to what was done for the fundamental group; again, the semistability of the end ensures indipendence from the choice of the base-ray.

\begin{example}
Let $M$ be a connected tame manifold without boundary, i.e. a manifold which is homeomorphic to the interior of a compact manifold $N$ with boundary. Moreover, let us suppose that $\partial N$ is connected. Then $M$ is a $1$-ended open manifold, and taking nested collar neighborhoods of the (missing) boundary as neighborhoods of infinity, we obtain that the fundamental group at infinity of $M$ is isomorphic to $\pi_1(\partial N)$.
\end{example}

The following definition is a weakening of the notion of tame manifold.
\begin{defn}{\label{inwtame}}
A manifold $M$ is \textit{inward tame} if for every neighborhood of infinity $U$ there exists a homotopy $H: U \times [0,1] \rightarrow U$ such that $H(\cdot, 0)=\id_U$ and $\overline{H(U, 1)}$ is compact.
\end{defn}
In other words, this means that every neighborhood of infinity $U$ can be shrunk into a compact subset within $U$; we remark that being inward tame implies that each clean neighborhood of infinity has finitely presented fundamental group \cite[Lemma 2.4]{GT03}. According to our definitions, it is immediate to check that a tame manifold is also inward tame. On the other hand, well known examples of inward tame manifolds which are not tame are the exotic universal covers of the aspherical manifolds produced by Davis in \cite{Dvs83} (see e.g. \cite[Example 3.3.32]{10.1007/978-3-319-43674-6_3}). Being inward tame ensures some nice properties for the topology at infinity of a manifold, as shown by the following proposition.

\begin{prop}[{\cite[Theorem 1.2]{GT03}}]\label{inwprop}
Let $M$ be an inward tame manifold; then 
\begin{itemize}
    \item $M$ has finitely many ends;
    \item each end has semistable fundamental group at infinity;
    \item each end has stable homology at infinity in all degrees.
\end{itemize}
\end{prop}
Hence, for inward tame manifolds the fundamental group at infinity at each end is well defined (i.e. it does not depend on the base-ray which determines the end or on the chosen sequence of neighborhoods of infinity).

\begin{defn}
We say that an end of an open manifold has amenable fundamental group at infinity if the fundamental group at infinity of the end is amenable, in the sense of topological groups. We remark that this condition is weaker then being amenable as a discrete group.
\end{defn}

In the introduction we defined manifolds simply connected at infinity. In the finitely-many-ended case, this definition can be read in terms of pro-triviality of the sequence that defines the fundamental group at infinity at each end.

\begin{lemma}{\cite[Proposition 3.4.36]{10.1007/978-3-319-43674-6_3}}\label{lemmscinf}
A finitely-many-ended manifold is simply connected at infinity if and only if the sequence which defines the fundamental group at infinity at every end is pro-trivial.
\end{lemma}

\begin{rem}
We remark that for a $1$-ended manifold having trivial fundamental group at infinity does not imply being simply connected at infinity: for example it is easy to see that the Whitehead manifold introduced in \cite{White} has trivial fundamental group at infinity but is not simply connected at infinity.
\end{rem}

We stress that a manifold simply connected at infinity need not be inward tame, not even in the 1-ended case (\cite[Example 3.5.11]{10.1007/978-3-319-43674-6_3}).

We state a result from \cite{Gu00} that ensures that, in dimension $n\geq 5$, a sequence $\{G_i, f_i\}$ pro-isomorphic to one that defines the fundamental group at infinity can be realized as a sequence of fundamental groups of neighborhoods of infinity. The result is originally stated for inward tame manifolds; however, the author remarks that it is enough to suppose that the fundamental groups of the neighborhoods of infinity are finitely presented (\cite[Section 6]{Gu00}).
\begin{prop}{\cite[Lemma 8]{Gu00}}\label{lemintinf}
Let $n \geq 5$, let $M$ be a $1$-ended open $n$-manifold such that every neighborhood of infinity has finitely presented fundamental group and let $\mathcal{G}=\{G_j, f_j\}_{j \in \mathbb{N}}$ be a sequence pro-isomorphic to a sequence that realizes the fundamental group at infinity of $M$. Then there exists a cofinal sequence of $0$-neighborhoods of infinity $\{U_j\}_{j \in \mathbb{N}}$ such that the inverse sequence $\pi_1(U_0) \longleftarrow \pi_1(U_1) \longleftarrow \pi_1(U_2) \longleftarrow \ldots$ is pro-isomorphic to $\mathcal{G}$.
\end{prop}

\begin{rem}\label{finmany}
The previous proposition can be generalized to the finitely-many-ended case. Let us suppose that $M$ has $k$ ends; by definition, there exists a compact subset $K_0 \subset M$ such that $M \setminus K_0$ has exactly $k$ unbounded connected components, so there exists a sequence of nested clean neighborhoods of infinity $U_0 \supset U_1 \supset \ldots$ such that each $U_j$ has exactly $k$ connected components $U_j^1, \ldots, U_j^k$ for each $j\in \mathbb{N}$. Let us suppose that for each $i=1, \ldots, k$ the sequence $\pi_1(U_0^i) \longleftarrow \pi_1(U_1^i) \longleftarrow \pi_1(U_2^i) \longleftarrow \ldots$ is proisomorphic to a sequence $\mathcal{G}^i=\{G_j^i, f_j^i\}_{j \in \mathbb{N}}$. We can apply Proposition \ref{lemintinf} to each end separately to obtain a sequence of (disconnected) neighborhoods of infinity $\{\widetilde{U_j}\}_{j \in \mathbb{N}}$ with the property that the inverse sequence of the fundamental groups of the $i$-th connected component of the $\widetilde{U_j}$'s is pro-isomorphic to the sequence $\mathcal{G}^i$.
\end{rem}

\section{Simplicial volume and amenable covers}\label{simvolamc:sec}

Let $X$ be a topological space and let $R$ be either the ring of the integers $\mathbb{Z}$ or of the real numbers $\mathbb{R}$. We denote by $S_n(X)$ the set of singular $n$-simplices on $X$ for a given natural number $n\in \mathbb{N}$. A singular $n$-chain on $X$ with $R$-coefficients is a sum $\sum_{i\in I}a_i \sigma_i$, where $I$ is a (possibly infinite) set of indices, $a_i \in R$ and the $\sigma_i$'s are singular simplices. A chain is said to be \textit{locally finite} if for every compact subset $K \subset X$ we have that there are only finitely many simplices of the chain whose image intersects $K$. We denote by $C_n^{\lf}(X; R)$ the module of singular locally finite $n$-chains on $X$ with $R$-coefficients. We observe that the obvious extension of the usual boundary operator sends locally finite chains to locally finite chains, hence it defines a boundary operator on $C_{*}^{\lf}(X; R)$; the homology of this complex with respect to this differential is denoted by $H_{*}^{\lf}(X;R)$ and is called the locally finite homology of $X$ with $R$-coefficients. We remark that, if $X$ is compact, then the locally finite homology coincides with the usual singular homology.

From now on, unless otherwise stated, when we omit the coefficients we will understand that $R=\mathbb{R}$.

We can endow the complex of locally finite chains $C_{*}^{\lf}(X)$ with the $\ell^1$-norm as follows:
\[||\sum_{i\in I}a_i \sigma_i||_1:=\sum_{i \in I}|a_i|.\]
We stress that a locally finite chain may have an infinite $\ell^1$-norm.

This norm induces a seminorm on $H_{*}^{\lf}(X)$, which we will again denote by $||\cdot||_1$.

From standard algebraic topology (see e.g. \cite[Theorem 5.4]{Lothesis}), we have that if $M$ is an oriented connected $n$-manifold, then $H_n^{\lf}(M, \mathbb{Z})\simeq \mathbb{Z}$, and this allows us to define the fundamental class $[M]_{\mathbb{Z}}\in H_n^{\lf}(M, \mathbb{Z})$ as a generator of this group. The image of $[M]_{\mathbb{Z}}$ under the change of coefficients map $H_n^{\lf}(M, \mathbb{Z}) \rightarrow H_n^{\lf}(M)$ is called the \textit{real} fundamental class and is denoted by $[M]\in H_n^{\lf}(M)$.

If $M$ is an oriented connected $n$-manifold, the \textit{simplicial volume} of $M$ is defined as
\[||M||=||[M]||_1.\]

We state two results from \cite{FM19} that give finiteness or vanishing conditions on the simplicial volume in terms of the existence of amenable covers with small multiplicity. First, we give some definitions.

We say that a subset $U$ of a topological space $X$ is amenable if $i_*(\pi_1(U,p))$ is an amenable subgroup of $\pi_1(X,p)$ for any choice of basepoint $p \in U$, where $i: U \hookrightarrow X$ denotes the inclusion. Following \cite{FM19}, we give the following definition.
\begin{defn}\label{amenatinfty}
A sequence of subsets $\{V_j\}_{j \in \mathbb{N}}$ is \textit{amenable at infinity} if there exists a sequence of neighborhoods of infinity $\{U_j\}_{j\in \mathbb{N}}$ such that
\begin{itemize}
    \item $\{U_j\}_{j \in \mathbb{N}}$ is locally finite;
    \item $V_j \subset U_j$ for each $j \in \mathbb{N}$;
    \item there exists $j_0 \in \mathbb{N}$ such that for every $j \geq j_0$ it holds that $V_j$ is an amenable subset of $U_j$.
\end{itemize}
\end{defn}

\begin{thm}[{\cite[Corollary 12]{FM19}}]\label{teoameninft}
Let $M$ be an open oriented triangulable manifold of dimension $m$, let $W$ be a neighborhood of infinity and let $\mathcal{U}=\{U_j\}_{j \in \mathbb{N}}$ be an open cover of $W$ such that each $U_j$ is relatively compact in $M$. Let us suppose also that $\mathcal{U}$ is amenable at infinity and that $\mult(\mathcal{U})\leq m$; then $||M||< + \infty$.
\end{thm}
\begin{thm}[{\cite[Corollary 11]{FM19}}]\label{teoamen}
Let $M$ be an open oriented triangulable manifold of dimension $m$ and let $\mathcal{U}=\{U_j\}_{j \in \mathbb{N}}$ be an amenable open cover of $M$ (i.e. each $U_j$ is an amenable subset of $M$) such that each $U_j$ is relatively compact in $M$. Let us suppose also that $\mathcal{U}$ is amenable at infinity and that $\mult(\mathcal{U})\leq m$; then $||M||=0$.
\end{thm}

\section{Proof of Theorem \ref{teor1}, Theorem \ref{teor2} and Theorem \ref{teor3}}\label{proof:sec}

\begin{proof}[Proof of Theorem \ref{teor1}]
First, let us deal with the 1-ended case. Since $M$ is inward tame, by Proposition \ref{inwprop} we have that $M$ has semistable fundamental group at infinity. Let $\mathcal{G}=\{G_0 \longleftarrow G_1 \longleftarrow \ldots\}$ be a sequence whose maps are all surjective which is pro-isomorphic to a sequence that realizes the fundamental group at infinity of $M$. Let $\{U_j\}_{j \in \mathbb{N}}$ be a cofinal sequence of clean $0$-neighborhoods of infinity given by Proposition \ref{lemintinf}, i.e. such that $\{n_j\}_{j\in \mathbb{N}}$ is a strictly increasing sequence, $\pi_1(U_j) \simeq G_{n_j}$ and that the inclusions $U_{j+1} \hookrightarrow U_{j}$ induce surjective maps between the fundamental groups for each $j \in \mathbb{N}$. By Lemma \ref{lemamn}, we have that $\pi_1(U_j)$ is amenable for each $j \in \mathbb{N}$. Let $V_j:= U_j \setminus \inte(U_{j+1})$; for $j \geq 0$, $V_j$ is a connected codimension-$0$ submanifold of $M$ with two boundary components, which are $\partial U_j$ and $\partial U_{j+1}$.

For every $j \in \mathbb{N}$, let us fix a regular neighborhood $N_j$ in $M$ of the boundary component of $U_j$ and a homeomorphism $N_j \cong \partial U_j \times [-1,1]$ such that $\partial U_j \times [-1,0)$ is contained in $U_{j-1} \setminus U_j$ and $\partial U_{j+1} \times (0,1]$ is contained in $U_{j+1} \setminus U_{j+2}$. Now, for each $j \in \mathbb{N}$ let us set $\widetilde{V_j}:= V_j \bigcup \partial U_j \times (-1/2,0] \bigcup \partial U_{j+1} \times [0, 1/2)$. Moreover, let us set $\widetilde{V}_{-1}:= \inte (M \setminus U_0) \bigcup \partial U_0 \times [0, 1/2)$. We have that $\widetilde{V_j}$ is open and relatively compact, and that the open cover $\mathcal{V}=\{\widetilde{V_j}\}_{j={-1}}^{+\infty}$ has multiplicity $2$. Moreover, for each $j \in \mathbb{N}$ we have that $\widetilde{V_j}$ is contained in $\widetilde{U_j}:=U_j \bigcup \partial U_j \times (-1,0]$;
the $\widetilde{U_j}$'s are locally finite and homotopy equivalent to the $U_j$'s, so in particular they have an amenable fundamental group. Hence, the image of $\pi_1(\widetilde{V_j})$ in $\pi_1(\widetilde{U_j})$ under the map induced by the inclusion $\widetilde{V_j}\hookrightarrow \widetilde{U_j}$ is also amenable. It follows that $\mathcal{V}$ is amenable at infinity. We can conclude that $||M||<+ \infty$ thanks to Theorem \ref{teoameninft}.

If $M$ has $k$ ends and $k \in \mathbb{N}_{\geq2}$, thanks to Remark \ref{finmany} we can simply apply the same procedure to each end separately and get the same conclusion.
\end{proof}

\begin{proof}[Proof of Theorem \ref{teor2}]
In the 1-ended case, for each $j\in \mathbb{N} \cup \{-1\}$ let us take $\widetilde{V_j}$ as in the proof of Theorem \ref{teor1}: since $\pi_1(M)$ is amenable, we have that the $\widetilde{V_j}$'s form an amenable open cover with the same properties as above, so by Theorem \ref{teoamen} we have that $||M||=0$.

If $M$ has more than one end, the conclusion follows again from Remark \ref{finmany}.
\end{proof}

\begin{rem}
Our proofs of Theorem \ref{teor1} and Theorem \ref{teor2} work even if $M$ is a triangulable open manifolds of dimension $n\geq 5$ with finitely many ends, such that the sequence which defines the fundamental group at infinity at each end is semistable and with amenable fundamental group at infinity at each end.
\end{rem}

We extend Theorem \ref{teor1} and Theorem \ref{teor2} to the case of finitely-many-ended manifolds that are simply connected at infinity (which, as remarked in Section \ref{fgi:sec}, need not be inward tame).

\begin{proof}[Proof of Theorem \ref{teor3}]
As usual, we write the proof in the 1-ended case, noting that it can be generalized to the finitely-many-ended one thanks to Remark \ref{finmany}. We are going to prove that if $M$ is simply connected at infinity then the fundamental groups of neighborhoods of infinity of $M$ are finitely presented, so that we can apply Proposition \ref{lemintinf}. Let $U_0 \supset U_1 \supset U_2 \supset \ldots$ be a cofinal sequence of nested clean $0$-neighborhoods of infinity; by Lemma \ref{lemmscinf}, the induced sequence on the fundamental groups is pro-trivial, i.e. (up to taking a subsequence) it fits in a diagram of the following form: 

\begin{center}
\begin{tikzcd}
\pi_1(U_0) & & \arrow[ll, "\lambda_1"] \pi_1(U_1) \arrow[dl, ""] & & \pi_1(U_2) \arrow[ll, "\lambda_2"] \arrow[dl, ""] &  \ldots \arrow[l] \\
& 1 \arrow[ul] & & \arrow[ll] 1 \arrow[ul] & & 1 \arrow[ll, ""] \arrow[ul] & \ldots \arrow[l]
\end{tikzcd} 
\end{center}
where $\lambda_j: \pi_1(U_{j+1})\rightarrow \pi_1(U_j)$ is the map induced by the inclusion $U_{j+1} \hookrightarrow U_j$ which, by commutativity of the diagram, is the trivial map. Let us fix $j\in \mathbb{N}$, let $K$ be the compact subset $U_j \setminus \inte U_{j+1}$. By Van Kampen Theorem, we have that \[\pi_1(U_j) \simeq \pi_1(K) *_{\pi_1(\partial U_{j+1})} \pi_1(U_{j+1}).\]
Let $\mu$ denote the map induced by the inclusion $\partial U_{j+1} \hookrightarrow K$. Since the elements of $\pi_1(U_{j+1})$ are trivial in the amalgamated product, $\pi_1(U_j)$ coincides with $\pi_1(K)/\mu(\pi_1(\partial U_{j+1}))$ (we can simplify the standard presentation of the amalgamated product by cancelling the generators and the relations of $\pi_1 (U_{j+1})$). From the compactness of $K$ and $\partial U_{j+1}$, it follows that $\pi_1(U_j)$ is a quotient of a finitely presentable group by a finitely presentable group, and so is finitely presentable. Therefore we can apply Proposition \ref{lemintinf} to obtain a sequence of simply connected neighborhoods of infinity, and then we can conclude that $||M||<+\infty$ as in the proof of Theorem \ref{teor1}. If in addition $\pi_1(M)$ is amenable, the vanishing of $||M||$ is deduced as in the proof of Theorem \ref{teor2}.
\end{proof}

\bibliography{biblio_amn}
\bibliographystyle{alpha}

\end{document}